\documentclass{amsart}
\newcommand{\be}{\begin{equation}}
\newcommand{\ee}{\end{equation}}
\newcommand{\bea}{\begin{eqnarray}}
\newcommand{\eea}{\end{eqnarray}}

\newcommand{\bee}{\begin{eqnarray*}}
\newcommand{\eee}{\end{eqnarray*}}

 \newtheorem{thm}{Theorem}[section]
 
 \newtheorem{lem}[thm]{Lemma}
 
 \theoremstyle{definition}

 \theoremstyle{remark}

\usepackage{enumitem}
 \usepackage{color}

 \title [Super-biderivations and linear super-commuting maps on the
Hom-Lie superalgebras]{linear super-commuting maps and Super-biderivations on the Hom-Lie superalgebras}
 \author[A. Pandey]{ Ashutosh Pandey }
\address{A. Pandey, Department of Mathematics and Statistics, Manipal University Jaipur, Jaipur, Rajasthan 303007, India.\\ Orcid Id 0000-0002-0002-3312}
\email{ashutoshpandey064@gmail.com}

\begin{document}
\maketitle
\begin{abstract}
This paper investigates the fundamental connections between linear super-commuting maps, super-biderivations, and centroids in Hom-Lie superalgebras under certain conditions. Our work generalizes the results of Bresar and Zhao on Lie algebras.
\end{abstract}
\section{Introduction}
Throughout the article, we work on a algebrically closed field $\mathbb{K}$ with characteristics different from 2. The investigation of biderivations and commuting maps originates in the field of associative ring theory, as noted in \cite{Mbre, Mbre1, ECposner}. A comprehensive overview of these maps and their applications, particularly in the context of Lie theory, was presented in \cite{Mbres}. Over the past ten years, there has been significant interest in exploring biderivations and similar maps in the setting of Lie algebras and superalgebras. Numerous studies have contributed to this growing body of research (see, for instance, \cite{Mbre2, Zchen,Xwang, Xuguo, Xita-Cxi})

The study of Hom-Lie structures is motivated by their applications in physics and the deformation theory of Lie algebras, particularly those related to Lie algebras of vector fields. Notable examples include Virasoro and $q$-deformations of Witt algebras, as constructed in \cite{Naiz, Rchar,Nhu}. The concept of a Hom-Lie algebra was first introduced in \cite{Jtha} and later extended to include quasi-Hom-Lie and quasi-Lie algebras in \cite{Dlar, Dlar1}. The construction and properties of   Chevalley-Eilenberg type homology and enveloping algebras for Hom-Lie algebras are discussed in \cite{Dya, Dya1}. deformation and cohomology theories for these algebras were addressed in \cite{Famm1, Ama}, while representation theory, particularly for multiplicative Hom-Lie algebras where the twisting map is an algebra homomorphism, is studied in \cite{Yshe}. A $q$-deformed version of the centerless W(2,2) Lie algebra, introduced as a Hom-Lie algebra in \cite{Lyau}. The introduction of Hom-Lie superalgebras, including the construction of the q-deformed Witt superalgebra, took place in \cite{Famm}. Its cohomology theory and the second cohomology of the q-deformed Witt superalgebra were investigated in \cite{Famm2, Famm3}. Additionally, Hom-Lie color algebras, introduced in \cite{Lyua1}.

The objective of this article is to explore the interactions between linear super-commuting maps, super-biderivations, and centroids in Hom-Lie superalgebras that meet specific criteria. This investigation goes beyond mere generalization, as we anticipate that linear commuting maps on Lie superalgebras will play a significant role in the development of the theory of functional identities for these structures. This expectation stems from the fact that studying linear commuting maps on Hom-Lie algebras serves as a foundational case for advancing the theory of functional identities in Hom-Lie algebras. The origin of this line of inquiry can be traced back to additive commuting maps on prime rings \cite{Naiz}, which eventually contributed to the establishment of the theory of functional identities in noncommutative rings \cite{Famm}.

\section{Notations and definitions}
In this section, for the reader’s ease, we will recall some fundamental facts about Hom-Lie superalgebras. For simplicity, the degree of an element 
$\zeta$ or a linear map $\phi$ is represented by 
$|\zeta|$ or $|\phi|$, respectively.

\textbf{Definition 2.1:} \cite{Lyuan} A Hom-Lie superalgebra is a superalgebra $(H, [\cdot, \cdot], \alpha)$, where $\alpha$ is a linear map on $H$ and $H = H_{\bar{0}} \oplus H_{\bar{1}}$, with a multiplication $[ \cdot, \cdot ]$ that satisfies the following two axioms:\\

     \textit{Skew-supersymmetry:} $[\zeta, \xi] = -(-1)^{|\zeta||\xi|} [\xi, \zeta]$,\\
     
 \textit{Hom-super Jacobi identity:} \begin{equation}\label{eq1a}
    (-1)^{|\zeta||\xi|}[\alpha(\zeta), [\xi, \omega]] + (-1)^{|\omega||\xi|}[\alpha(\omega), [\zeta, \xi]] + (-1)^{|\xi||\zeta|}[\alpha(\xi), [\omega, \zeta]] = 0\end{equation}
for all homogeneous element $\zeta, \xi, \omega \in H$. Hom-Lie superalgebra $H$ is said to be multiplicative if $\alpha([\zeta, \xi])=[\alpha(\zeta),\alpha(\xi)]$ for every $\zeta, \xi \in H$.

Let \( H = H_{\bar{0}} \oplus H_{\bar{1}} \) be a superalgebra over a field \( \mathbb{K} \). A linear map \( d: H \to H \) is said to be a \textit{homogeneous linear map of degree} \( s \) (i.e., $|d|=s$) if it satisfies \( d(H_i) \subseteq H_{i+s} \) for all \( i \in \mathbb{Z}_2 \), where \( s \in \mathbb{Z}_2 \). If \( s = 0 \) then \( d \) will be termed as \textit{even linear map}. If $|d|$ occurs for a linear map $d$ then $d$ implies a homogeneous map. If $|\zeta|$ occurs for a Hom-Lie superalgebra \( H = H_{\bar{0}} \oplus H_{\bar{1}} \), then \( \zeta \) will be a homogeneous element in \( H \).

Now, for a Hom-Lie superalgebra \( H = H_{\bar{0}} \oplus H_{\bar{1}} \), a bilinear map \( \phi: H \times H \to H \) is called a \textit{homogeneous bilinear map of degree} \( s \) if it satisfies \( \phi(H_i, H_j) \subseteq H_{i+j+s} \) for all \( i, j \in \mathbb{Z}_2 \), where \( |\phi| = s \). 

\textbf{Definition 2.2:}\cite{Lyuan} Let $(H= H_{\bar{0}} \oplus H_{\bar{1}}, [.,.], \alpha)$ be a Hom-Lie superalgebra. A bilinear map \( \phi: H \times H \to H \) is referred to as a \textit{super-biderivation} of \( H \) if $\phi \alpha=\alpha \phi$ and for all \( \zeta, \xi, \omega \in H \), it satisfies:
\[
\phi(\zeta, \xi) = (-1)^{|\zeta||\xi|} \phi(\xi, \zeta),
\]
and the super-biderivation property:
\[
\phi([\zeta, \xi], \alpha(\omega)) = (-1)^{|\phi||\zeta|}[\alpha(\zeta), \phi(\xi, \omega)] + (-1)^{|\xi||\omega|} [\phi(\zeta, \omega), \alpha(\xi)].
\]
The set of all super-biderivations of degree \( s \) for \( H \) is denoted by \( \text{BDer}_s(H) \). Clearly, $$ \text{BDer}(H) = \text{BDer}_0(H) \oplus \text{BDer}_1(H).$$

\textbf{Definition 2.3:} For a Hom-Lie superalgebra \( H = H_{\bar{0}} \oplus H_{\bar{1}} \), an \textit{even linear map} \( d: H \to H \) is called a \textit{linear super-commuting map} if $d\alpha=\alpha d$ and for all \( \zeta, \xi \in H \), it satisfies:
\[
[d(\zeta), \xi] = [\zeta, d(\xi)].
\]

\section{Super-biderivations}
Throughout the proofs of Lemmas and Theorems, we will use the notation \( H \) to represent the Hom-Lie superalgebra \( (H, [\cdot, \cdot], \alpha) \), where \( \alpha \) is an automorphism of \( H \).

\begin{lem}\label{lem3.1} Let \( H \) be a Hom-Lie superalgebra and \( d: H \to H \) be a linear super-commuting map. Assume 
\[
\phi(\zeta, \xi) = [\zeta, d(\xi)], \quad \forall \zeta, \xi \in H.
\]
Then \( \phi: H \times H \to H \) is a super-biderivation.
\end{lem}
\begin{proof} Since $d$ is even and $\phi(\zeta, \xi) = [\zeta, d(\xi)]$, thus we get $ |\phi|=|d|=\bar{0}$. 
Again, 
   \begin{align*}
(-1)^{|\xi| |\zeta|}\phi(\xi, \zeta) &= (-1)^{|\xi| |\zeta|} [\xi, d(\zeta)] \\
&= (-1)^{|\xi| |\zeta|+|\xi|(|d|+|\zeta|)}  [d(\zeta), \xi]  \\
&= [d(\zeta), \xi] \\
&= [\zeta, d(\xi)] \\
&= \phi(\zeta, \xi).
\end{align*}
Again, \begin{align*}
\phi(\alpha(\xi), \alpha(\zeta)) &= [\alpha(\xi), d(\alpha(\zeta))] \\
&= [\alpha(\xi), \alpha(d(\zeta))] \\
&= \alpha([\xi, d(\zeta)]) \\
&= \alpha(\phi(\xi, \zeta)).
\end{align*}
Now, from hypothesis and using $d\alpha=\alpha d$, we have:
\begin{align*}
 \phi([\zeta, \xi],\alpha(\chi))=\big[[\zeta, \xi],d(\alpha(\chi))\big]= \big[[\zeta, \xi],\alpha(d(\chi))\big].
\end{align*} By using the super Hom-Jacobi identity, we obtain:
\begin{align*}
    \big[[\zeta, \xi],\alpha(d(\chi))\big] &= (-1)^{|\xi||\chi|}\big[[\zeta,d(\chi)],\alpha(\xi)\big] + \big[\alpha(\zeta), [\xi, d(\chi)]\big] \\
&= (-1)^{|\chi||\xi|}\big[[\zeta,d(\chi)],\alpha(\xi)\big] + (-1)^{|\phi||\zeta|}\big[\alpha(\zeta), [\xi, d(\chi)]\big] \\
&= (-1)^{|\chi||\xi|}\big[\phi(\zeta,\chi),\alpha(\xi)\big] + (-1)^{|\phi||\zeta|} \big[\alpha(\zeta), \phi(\xi, \chi)\big].
\end{align*}

Thus, $$\phi([\zeta, \xi],\alpha(\chi))= (-1)^{|\chi||\xi|}\big[\phi(\zeta,\chi),\alpha(\xi)\big] + (-1)^{|\phi||\zeta|} \big[\alpha(\zeta), \phi(\xi, \chi)\big].$$
It follows that $\phi$ is a super-biderivation.
\end{proof}
\textbf{Definition 3.2.} Let \( H \) be a Hom-Lie superalgebra, and let \( \delta: H \to H \) be a linear map. The centroid of \( H \) is defined as:
\[
C(H) = \left\{ \delta : H \to H \mid \delta([\zeta, \xi]) = (-1)^{|\delta||\zeta|} [\alpha(\zeta), \delta(\xi)], \text{and} \ \alpha \delta=\delta \alpha   \, \forall \zeta, \xi \in H \right\}.
\]
Denote by \( C_s(H) \) the set of all elements of degree \( s \) in \( C(H) \). Clearly, we have:
\[
C(H) = C_{\bar{0}}(H) \oplus C_{\bar{1}}(H).
\]
\begin{lem}\label{Lemma 3.3.} Let \( H \) be a Hom-Lie superalgebra. Suppose \( \delta: H \to H \) is a linear map, and \( \phi: H \times H \to H \) is a bilinear map. Define \( \phi(\zeta, \xi) = \alpha^{-1}\delta([\zeta, \xi]) \). If \( \delta \in C(H) \), then \( \phi \) is a super-biderivation.
\end{lem}
\begin{proof}
From definition 3.2, we have
\begin{align*}
    \phi(\zeta, \xi)=\delta([\zeta, \xi])=(-1)^{|\delta||\xi|}[\alpha(\zeta), \delta(\xi)]
\end{align*} for all $\zeta, \xi \in H$. Thus, $|\phi|=|\delta|$.
First, we get:
\begin{align*}
(-1)^{|\xi||\zeta|} \phi(\xi, \zeta) &= (-1)^{|\xi||\zeta|} \delta([\xi, \zeta]) \\
&= (-1)^{|\xi||\zeta|} \delta((-1)^{|\xi||\zeta|} [\zeta, \xi]) \\
&= \phi(\zeta, \xi), \quad \forall \, \zeta, \xi \in H.
\end{align*}
Also, \begin{align*}
\phi(\alpha(\zeta), \alpha(\xi)) &= \delta([\alpha(\zeta), \alpha(\xi)]) \\
&= \delta \alpha([\zeta, \xi]) \\
&= \alpha(\delta([\zeta, \xi])) \\
&= \alpha(\phi(\zeta, \xi)).
\end{align*}
Second by super Hom-Jacobi identity, we get
\begin{align*}
\phi([\zeta, \xi], \alpha(\chi)) &= \alpha^{-1}\delta([[ \zeta, \xi ], \alpha(\chi)]) \\
&= \alpha^{-1}\delta([\alpha(\zeta), [\xi, \chi]] + (-1)^{|\zeta||\xi|} [\alpha(\xi), [\zeta, \chi]]) \\
&= \alpha^{-1}\delta([\alpha(\zeta), [\xi, \chi]]) + (-1)^{|\zeta||\xi|} \alpha^{-1}\delta([\alpha(\xi), [\zeta, \chi]]) \\
&= (-1)^{|\zeta||\delta|} \alpha^{-1}[\alpha^2(\zeta), \delta([\xi, \chi])] + (-1)^{|\zeta||\xi| + |\xi||\delta|} \alpha^{-1}[\alpha^2(\xi), \delta([\zeta, \chi])] \\
&= (-1)^{|\zeta||\delta|} \alpha^{-1}[\alpha^2(\zeta), \delta([\xi, \chi])] \\
&\quad + (-1)^{|\zeta||\xi| + |\xi|(|\delta| + |\zeta| + |\chi|)} \alpha^{-1}[\delta([\zeta, \chi]), \alpha^2(\xi)] \\
&= (-1)^{|\zeta||\delta|} [\alpha(\zeta), \alpha^{-1}\delta([\xi, \chi])] + (-1)^{|\xi||\chi|} [\alpha^{-1}\delta([\zeta, \chi]), \alpha(\xi)] \\
&= (-1)^{|\zeta||\delta|} [\alpha(\zeta), \phi(\xi, \chi)] + (-1)^{|\xi||\chi|} [\phi(\zeta, \chi), \alpha(\xi)], \quad \forall \zeta, \xi, \chi \in H.
\end{align*}
Thus, $\phi$ is a super-biderivation.
\end{proof}
\begin{lem}\label{Lemma 3.4} Let \( \phi \) be a super-biderivation on Hom-Lie superalgebra \( H \), then
\begin{align*}
[\phi(\zeta, \xi), [\pi, \tau]] &= (-1)^{|\phi|(|\zeta|+|\xi|)} [[\zeta, \xi], \phi(\pi, \tau)]
\end{align*}
for any homogeneous \( \zeta, \xi, \pi, \tau \in H \).
\end{lem}
\begin{proof}
Firstly, we evaluate $\phi([\alpha(\zeta), \alpha(\pi)], [\alpha(\xi), \alpha(\tau)])$ in two different ways using the definition of super-biderivation.
On the one hand, one has
\begin{align*}
\phi([\alpha([\zeta, \pi]), [\alpha(\xi), \alpha(\tau)]])
&= [\phi([\zeta, \pi], \alpha(\xi)), \alpha^2(\tau)] 
+ (-1)^{|\xi|(|\phi| + |\zeta| + |\pi|)} [\alpha^2(\xi), \phi([\zeta, \pi], \alpha(\tau))] \\
&= (-1)^{|\phi||\zeta|} [[\alpha(\zeta), \phi(\pi, \xi)], \alpha^2(\tau)] 
+ (-1)^{|\pi||\xi|} [[\phi(\zeta, \xi), \alpha(\pi)], \alpha^2(\tau)] \\
&\quad + (-1)^{|\xi|(|\phi| + |\zeta| + |\pi|)} \left( (-1)^{|\phi||\zeta|} [\alpha^2(\xi), [\alpha(\zeta), \phi(\pi, \tau)]] \right. \\
&\quad \left. + (-1)^{|\pi||\tau|} [\alpha^2(\xi), [\phi(\zeta, \tau), \alpha(\pi)]] \right).
\end{align*}
On the other hand, we have
\begin{align*}
\phi([\alpha(\zeta), \alpha(\pi)], \alpha([\xi, \tau])) &= (-1)^{|\phi||\zeta|} [\alpha^2(\zeta), \phi(\alpha(\pi), [\xi, \tau])] + (-1)^{|\pi|(|\xi|+|\tau|)} [\phi(\alpha(\zeta), [\xi, \tau]), \alpha^2(\pi)] \\
&= (-1)^{|\phi||\zeta|} [\alpha^2(\zeta), [\phi(\pi, \xi), \alpha(\tau)]] + (-1)^{(|\phi|+|\pi|)|\xi|} [\alpha^2(\zeta), [\alpha(\xi), \phi(\pi, \tau)]] \\
&\quad + (-1)^{|\pi|(|\xi|+|\tau|)} [[\phi(\zeta, \xi), \alpha(\tau)], \alpha^2(\pi)] + (-1)^{(|\phi|+|\zeta|)|\xi|} [[\alpha(\xi), \phi(\zeta, \tau)], \alpha^2(\pi)].
\end{align*}
Comparing the above two equations and using the super Hom-Lie Jacobi identity, we get:
\begin{align}\label{eqpg1}
    &\big[\phi(\alpha(\zeta), \alpha(\xi)), [\alpha(\pi), \alpha(\tau)]\big] 
    - (-1)^{|\phi|(|\zeta|+|\xi|)} \big[[\alpha(\zeta), \alpha(\xi)], \phi(\alpha(\pi), \alpha(\tau))\big] \nonumber \\
    &= (-1)^{|\pi||\tau| + |\tau||\xi| + |\xi||\pi|} 
    \bigg(\big[\phi(\alpha(\zeta), \alpha(\tau)), [\alpha(\xi), \alpha(\pi)]\big] \nonumber \\
    &\quad - (-1)^{|\phi|(|\zeta|+|\tau|)} 
    \big[[\alpha(\zeta), \alpha(\tau)], \phi(\alpha(\xi), \alpha(\pi))\big]\bigg).
\end{align}
Now let 
\[
\Psi(\zeta, \xi; \pi, \tau) = \big[\phi(\alpha(\zeta), \alpha(\xi)), [\alpha(\pi), \alpha(\tau)]\big] 
    - (-1)^{|\phi|(|\zeta|+|\xi|)} \big[[\alpha(\zeta), \alpha(\xi)], \phi(\alpha(\pi), \alpha(\tau))\big].
\]
According to the equality (\ref{eqpg1}), we get that
\[
\Psi(\zeta, \xi; \pi, \tau) = (-1)^{|\pi||\tau| + |\tau||\xi| + |\xi||\pi|} \Psi(\zeta, \tau; \pi, \xi).
\]
For one thing, we get
\begin{align*}
\Psi(\zeta, \xi; \pi, \tau) &= -( - 1)^{|\pi||\tau|} \Psi(\zeta, \xi; \tau, \pi) \\
&= -( - 1)^{|\pi||\tau|} ( - 1)^{|\pi||\tau| + |\tau||\xi| + |\xi||\pi|} \Psi(\zeta, \pi; \tau, \xi) \\
&= ( - 1)^{|\xi||\pi|} \Psi(\zeta, \pi; \xi, \tau).
\end{align*}
For another case, we also have:
\begin{align*}
\Psi(\zeta, \xi; \pi, \tau) &= ( - 1)^{|\pi||\tau| + |\tau||\xi| + |\xi||\pi|} \Psi(\zeta, \tau; \pi, \xi) \\
&= -( - 1)^{|\pi||\tau| + |\tau||\xi| + |\xi||\pi|} ( - 1)^{|\xi||\pi|} \Psi(\zeta, \tau; \xi, \pi) \\
&= -( - 1)^{|\xi||\pi|} ( - 1)^{2(|\pi||\tau| + |\tau||\xi| + |\xi||\pi|)} \Psi(\zeta, \pi; \xi, \tau) \\
&= -( - 1)^{|\xi||\pi|} \Psi(\zeta, \pi; \xi, \tau).
\end{align*}
Thus, it implies that:
\[
\Psi(\zeta, \xi; \pi, \tau) = -\Psi(\zeta, \xi; \pi, \tau).
\]
Since the base field is different from characteristics 2, this forces:
\[
\Psi(\zeta, \xi; \pi, \tau) = 0,
\]
i.e.
\[
[\phi(\alpha(\zeta), \alpha(\xi)), [\alpha(\pi), \alpha(\tau)]] = ( - 1)^{|\phi|(|\zeta| + |\xi|)} [[\alpha(\zeta), \alpha(\xi)], \phi(\alpha(\pi), \alpha(\tau))]
\]
which implies that $$[\phi(\zeta, \xi), [\pi, \tau]] = ( - 1)^{|\phi|(|\zeta| + |\xi|)} [[\zeta, \xi], \phi(\pi, \tau)].$$
\end{proof}
\begin{lem}\label{Lemma 3.5:} If $|\zeta| + |\xi| = \bar{0}$, then $[\phi(\zeta, \xi), [\zeta, \xi]] = 0$ for any homogeneous $\zeta, \xi \in H$.
\end{lem}
\begin{proof}
Substituting $\pi = \zeta$ and $\tau = \xi$ in Lemma (\ref{Lemma 3.4}), we obtain
\begin{align*}
    [\phi(\zeta, \xi), [\zeta, \xi]] &= (-1)^{|\phi|(|\zeta|+|\xi|)} [[\zeta, \xi], \phi(\zeta, \xi)] \\
    &= -(-1)^{2|\phi|(|\zeta|+|\xi|) + 2|\zeta||\xi| + |\zeta|^2 + |\xi|^2} [\phi(\zeta, \xi), [\zeta, \xi]] \\
    &= -(-1)^{|\zeta|^2 + |\xi|^2} [\phi(\zeta, \xi), [\zeta, \xi]] \\
    &= -(-1)^{|\zeta| + |\xi|} [\phi(\zeta, \xi), [\zeta, \xi]].
\end{align*}
In view of $|\zeta| + |\xi| = 0$, then we have
\[
[\phi(\zeta, \xi), [\zeta, \xi]] = -[\phi(\zeta, \xi), [\zeta, \xi]].
\]
Therefore, it follows that
\[
[\phi(\zeta, \xi), [\zeta, \xi]] = 0.
\] 
\end{proof}
\begin{lem}\label{Lemma 3.6:} Let $\phi$ be a super-biderivation on $H$. If $[\zeta, \xi] = 0$, then $\phi(\zeta, \xi) \in Z([H, H])$, where $Z([H, H])$ is the center of $[H, H]$.
\end{lem}
\begin{proof} If $[\zeta, \xi] = 0$, then we get
\begin{align*}
    [\phi(\zeta, \xi), [\pi, \tau]] &= (-1)^{|\phi|(|\zeta| + |\xi|)} [[\zeta, \xi], \phi(\pi, \tau)] \\
    &= 0,
\end{align*}
for every $\pi, \tau \in H$.

Hence, we obtain that $\phi(\zeta, \xi)$ commutes with $[H, H]$, i.e., $\phi(\zeta, \xi) \in Z([H, H])$.
\end{proof}
\textbf{Definition:} Let \( X \) be a non-empty subset of the Hom-Lie superalgebra \( H \). Define the centralizer of \( X \) in \( H \) as:
\[
Z_H(X) = \{ v \in H \mid [X, v] = 0 \}.
\]
If \( X = H \), then \( Z(H) = Z_H(H) \), which represents the center of \( H \). If \( Z_H(H) = \{0\} \), then \( H \) is called centerless. Let \( H' = [H, H] \) denote the derived algebra of \( H \). If \( H = H' \), then \( Z_H(H') = Z_{H'}(H') \), representing the center of \( H' \).

\begin{lem}\label{lem 3.6}
  Let \( H \) be a Home-Lie superalgebra and \( \phi: H \times H \to H \) is a super-biderivation on \( H \). Then for all \( \zeta, \xi, \omega \in H \), the following holds: 
  \[
\phi\left(\alpha(\omega), [\zeta, \xi]\right) - (-1)^{|\phi||\omega|} [\alpha(\omega), \phi(\zeta, \xi)] \in Z_H(H'),
\]
\end{lem}
\begin{proof}
From Lemma 3.5, we obtain
\begin{equation} \label{1''}
    \big[\phi(\alpha(\pi), \alpha(\omega)), [\alpha(\zeta), \alpha(\xi)]\big] = (-1)^{|\phi|(|\pi| + |\omega|)} \big[[\alpha(\pi), \alpha(\omega)], \phi(\alpha(\zeta), \alpha(\xi))\big].
\end{equation}
Now, replacing \(\zeta\) by \([\zeta, \omega]\) and \(\xi\) by \(\alpha(\xi)\), we have
\begin{equation}\label{eqpg2}
    \big[\phi(\alpha(\pi), \alpha(\omega)), [[\alpha(\zeta), \alpha(\omega)], \alpha^2(\xi)]\big] 
    = (-1)^{|\phi|(|\pi| + |\omega|)} \big[[\alpha(\pi), \alpha(\omega)], \phi(\alpha([\zeta, \omega]), \alpha^2(\xi))\big].
\end{equation}
By the Hom-super-Jacobi identity, we also have
\begin{equation}\label{eqpg3'}
    \begin{aligned}
        &\big[\phi(\alpha(\pi), \alpha(\omega)), [[\alpha(\zeta), \alpha(\omega)], \alpha^2(\xi)]\big] \\
        &= \big[\phi(\alpha(\pi), \alpha(\omega)), [\alpha^2(\zeta), [\alpha(\omega), \alpha(\xi)]]\big] \\
        &\quad + (-1)^{|\zeta||\omega|} \big[\phi(\pi, \omega), [\alpha^2(\omega), [\alpha(\zeta), \alpha(\xi)]]\big]
    \end{aligned}
\end{equation}
Thus, by combining these with Equation (\ref{eqpg2}),
\begin{align}\label{eqpg3}
    &\big[\phi(\alpha(\pi), \alpha(\omega)), [[\alpha(\zeta), \alpha(\omega)], \alpha^2(\xi)]\big] \nonumber \\
    &= (-1)^{|\phi|(|\pi| + |\omega|)} \big[[\alpha(\pi), \alpha(\omega)], \phi(\alpha^2(\zeta), [\alpha(\omega), \alpha^2(\xi)])\big] \nonumber \\
    &\quad + (-1)^{|\zeta||\omega| + |\phi|(|\pi| + |\omega|)} \big[[\alpha(\pi), \alpha(\omega)], \phi(\alpha^2(\omega), [\alpha(\zeta), \alpha^2(\xi)])\big]. 
\end{align}
From Equation (\ref{eqpg3'}) and (\ref{eqpg3}), we get:
\begin{align*}
    &(-1)^{|\phi|(|\pi| + |\omega|)} \big[ [\alpha(\pi), \alpha(\omega)], \Big( \phi([\alpha(\zeta), \alpha(\omega)], \alpha^2(\xi)) \\
    &\quad - \phi(\alpha^2(\zeta), [\alpha(\omega), \alpha(\xi)]) + (-1)^{|\zeta||\omega|} \phi(\alpha^2(\omega), [\alpha(\zeta), \alpha(\xi)]) \Big) \big] = 0.
\end{align*} We know that $\phi$ is a super-biderivation, thus
\begin{align*}
    2(-1)^{|\phi|(|\pi|+|\omega|) + |\zeta||\omega|} 
    &\Big[ [\alpha(\pi), \alpha(\omega)], (-1)^{|\omega||\xi|+|\omega||\zeta|}[\phi(\alpha(\zeta), \alpha(\xi)), \alpha^2(\omega)]  \\
    &+ (-1)^{|\phi||\zeta| + |\omega||\zeta|} 
     [\alpha^2(\zeta), \phi(\alpha(\omega), \alpha(\xi))]  \\
    &+ [ \phi(\alpha(\omega), \alpha(\zeta)), \alpha^2(\xi) ]\Big]=0.
\end{align*}
Since 
$$(-1)^{|\phi||\zeta| + |\omega||\zeta|} 
     [\alpha^2(\zeta), \phi(\alpha(\omega), \alpha(\xi))] 
    + [ \phi(\alpha(\omega), \alpha(\zeta)), \alpha^2(\xi)]=\phi(\alpha^2(\omega), [\alpha(\zeta), \alpha(\xi)]),$$
it implies that
$$\Big[[\alpha(\pi), \alpha(\omega)], -(-1)^{|\omega||\phi|}[\alpha^2(\omega),\phi(\alpha(\zeta), \alpha(\xi))]+\phi(\alpha^2(\omega), [\alpha(\zeta), \alpha(\xi)])\Big]=0.$$
Again, since $\phi \alpha=\alpha \phi$ thus
$$-(-1)^{|\omega||\phi|}[\alpha(\omega), \phi(\zeta,\xi)]+\phi(\alpha(\omega),[\zeta,\xi])\in Z_{H}(H').$$
\end{proof}
\begin{lem}\label{lem3.7}
Let \( (H, \alpha) \) is a perfect Hom-Lie superalgebra and $\phi: H \times H\longrightarrow H$ is a super-biderivation, then
\[
\phi(\omega, [\zeta, \xi]) = (-1)^{|\phi||\omega|}[\omega, \phi(\zeta, \xi)], \quad \forall \ \zeta, \xi, \omega \in H.
\]
\end{lem}
\begin{proof}
We have
\begin{align*}
&\phi([\alpha(\pi), \alpha(\omega)], [\alpha(\zeta), \alpha(\xi)]) 
- (-1)^{|\phi|(|\pi| + |\omega|)} [\alpha(\pi), \alpha(\omega)], \phi(\alpha(\zeta), \alpha(\xi)) \\
&= (-1)^{|\phi||\pi|} \big[\alpha^2(\pi), \phi(\alpha(\omega), [\zeta, \xi])\big] 
+ (-1)^{|\omega|(|\zeta| + |\xi|)} \big[\phi(\alpha(\pi), [\zeta, \xi]), \alpha^2(\omega)\big] \\
&\quad - (-1)^{|\phi|(|\pi| + |\omega|)} \big[\alpha^2(\pi), [\alpha(\omega), \phi(\zeta, \xi)]\big] 
+ (-1)^{|\phi|(|\pi| + |\omega|) + |\pi||\omega|} \big[\alpha^2(\omega), [\alpha(\pi), \phi(\zeta, \xi)]\big] \\
&= (-1)^{|\phi||\pi|} \Big[\alpha^2(\pi), \phi(\alpha(\omega), [\zeta, \xi]) 
- (-1)^{|\phi||\omega|} [\alpha(\omega), \phi(\zeta, \xi)]\Big] \\
&\quad + (-1)^{|\omega|(|\zeta| + |\xi|) + |\omega|(|\phi| + |\pi| + |\zeta| + |\xi|)} \big[\alpha^2(\omega), \phi(\alpha(\pi), [\zeta, \xi])\big] \\
&\quad + (-1)^{|\phi|(|\pi| + |\omega|) + |\pi||\omega|} \big[\alpha(\omega), [\alpha(\pi), \phi(\zeta, \xi)]\big] \\
&= (-1)^{|\phi||\pi|} \Big[\alpha^2(\pi), \phi(\alpha(\omega), [\zeta, \xi]) 
- (-1)^{|\phi||\omega|} [\alpha(\omega), \phi(\zeta, \xi)]\Big] \\
&\quad - (-1)^{|\omega|(|\phi| + |\pi|)} \Big[\alpha^2(\omega), \phi(\alpha(\pi), [\zeta, \xi]) 
- (-1)^{|\phi||\pi|} [\alpha(\pi), \phi(\zeta, \xi)]\Big]
\end{align*} for all $\pi, \omega, \zeta, \xi \in H.$

By Lemma 3.6, $\phi\alpha=\alpha\phi$ and \(H = H'\), we have:
\[
\phi(\pi, [\zeta, \xi]) = (-1)^{|\phi||\pi|} [\pi, \phi(\zeta, \xi)], \quad \forall \pi, \zeta, \xi \in H.
\]

\end{proof}
\begin{thm}\label{thm1}
  Let \(H\) be a centerless Hom-Lie superalgebra with \(H = H'\). Then, every super biderivation \(\phi : H \times H \to H\) can be expressed as
\[
\phi(\zeta, \xi) = \alpha^{-1}\delta([\zeta, \xi]),
\]
for all $\zeta, \xi \in H$, where \(\delta\) is in the centroid of \(H\).
\end{thm}
\begin{proof}
    Suppose \( \delta : H \to H \) is a linear map given by
\begin{equation}\label{eq3.7}
\alpha^{-1}\delta([\zeta, \xi]) = \phi(\zeta, \xi)
\end{equation} for all $\zeta, \xi \in H$.
By Lemma 3.7, \(\delta\) is well-defined. In fact, suppose \( \sum_i [\pi_i, \omega_i] = 0 \), we have
\[
0 = \phi\left( \pi, \sum_i [\pi_i, \omega_i] \right) = \sum_i \phi(\pi, [\pi_i, \omega_i]) = (-1)^{|\pi||\phi|} \Big[\pi, \sum_i \phi(\pi_i, \omega_i)\Big] .
\]
Since \( Z(H) = \{0\} \), we have \( \sum_i \phi(\pi_i, \omega_i) = 0 \). Therefore, \( \sum_i \alpha\phi(\pi_i, \omega_i) = 0 \), which implies $\delta$ is well defined. Furthermore, suppose \( \omega'_i, \omega''_i \in H \) such that \( \omega = \sum_{i=1}^{k} [\omega'_i, \omega''_i] \), then by Equation (\ref{eq3.7})
\[
\delta([\pi, \omega])=\alpha\phi(\pi, \omega)=\phi(\alpha(\pi),\alpha(\omega)) = \phi\left(\alpha(\pi),\sum_{i=1}^{k} [\alpha(\omega'_i), \alpha(\omega''_i)]\right), \quad \forall \pi, \omega \in H.\]
By Lemma (\ref{eq3.7}), we get
\begin{align*}
    \delta([\pi, \omega]) &= (-1)^{|\phi||\pi|} \left[\alpha(\pi), \sum_{i=1}^{k} \phi(\alpha(\omega'_i), \alpha(\omega''_i)) \right] \\
    &= (-1)^{|\phi||\pi|} \left[\alpha(\pi), \sum_{i=1}^{k} \alpha \phi(\omega'_i, \omega''_i) \right] \\
    &= (-1)^{|\phi||\pi|} \left[\alpha(\pi), \sum_{i=1}^{k} \delta([\omega'_i, \omega''_i]) \right] \\
    &= (-1)^{|\phi||\pi|} \left[\alpha(\pi), \delta(\omega) \right].
\end{align*}
 for all $\omega, \pi \in H$. Thus, $\delta$ is in centroid of $H$.
\end{proof}
\textbf{Definition:}
 Let \( \phi \) be a skew-symmetric bilinear map such that \( \phi(H, H') = 0 \). In this case, \( \phi \) is a super-biderivation, referred to as a trivial super-biderivation of \( H \).
 
\textbf{Remark 3.10.} Let \( \phi : H \times H \to H \) be an arbitrary super-biderivation of \( H \). Then
\begin{align*}
    0 &= \phi([\omega, \zeta], \alpha(\xi)) \\
      &= (-1)^{|\phi||\omega|} [\alpha(\omega), \phi(\zeta, \xi)] 
      + (-1)^{|\zeta||\xi|} [\phi(\omega, \xi), \alpha(\zeta)] \\
      &= (-1)^{|\phi||\omega|} [\phi(\omega, \xi), \alpha(\zeta)]
\end{align*}
for all \( \zeta, \xi \in H \), \( \omega \in Z(H)\), which implies that \( [\phi(\omega, \xi), \alpha(\zeta)] = 0 \). Thus, \( \phi(Z(H), H) \subseteq Z(H) \). Moreover, let \( \bar{H} = H/Z(H) \) and 
\[
\bar{\phi}(\bar{\zeta}, \bar{\xi}) = \overline{\phi(\zeta, \xi)}, \quad \forall \zeta, \xi \in H.
\]
Then \( \bar{\phi} : \bar{H} \times \bar{H} \to \bar{H} \) is a super-biderivation, where \( \bar{\zeta} = \zeta + Z(H) \in \bar{H}, \ \forall \zeta \in H \).
\begin{thm}\label{thm3.12}
 Let \(H\) be a Hom-Lie superalgebra. Up to isomorphism, the mapping \(\phi \mapsto \bar{\phi}\) establishes a one-to-one correspondence between trivial superderivations \(\phi\) of \(H\) and trivial superderivations \(\bar{\phi}\) of \(\bar{H}\).
\end{thm}
\begin{proof} It is evident that the map \(\phi \mapsto \bar{\phi}\) is surjective. We now demonstrate it is injective. Assume \(\phi_1\) and \(\phi_2\) are superderivations of \(H\) such that \(\bar{\phi_1} = \bar{\phi_2}\). Let \(\phi = \phi_1 - \phi_2\). Then, 
\[
\bar{\phi_1}(\bar{\zeta}, \bar{\xi}) = \bar{\phi_2}(\bar{\zeta}, \bar{\xi}), \quad \forall \zeta, \xi \in H.
\]
Since
\[
\bar{\phi_1}(\bar{\zeta}, \bar{\xi}) = \overline{\phi_1(\zeta, \xi)} \quad \text{and} \quad \bar{\phi_2}(\bar{\zeta}, \bar{\xi}) = \overline{\phi_2(\zeta, \xi)},
\]
it follows that
\[
\overline{\phi_1(\zeta, \xi)} = \overline{ \phi_2(\zeta, \xi)}.
\]
Thus,
\[
\phi(\zeta, \xi) = \phi_1(\zeta, \xi) - \phi_2(\zeta, \xi) \in Z(H),
\]
which indicates that \(\phi\) is a trivial super-biderivation.
\end{proof}
\textbf{Definition 3.10.} Let \(H\) be a Hom-Lie superalgebra. A super-biderivation \(\phi : H \times H \to Z_{H}(H')\) is termed as special super-biderivation of \(H\) if it satisfies \(\phi(H', H') = 0\).\\

\textbf{Remark 3.11.} Let \(H\) be a Hom- Lie superalgebra, and \(\phi : H \times H \to H\) be a super-biderivation which satisfies condition
\begin{equation}\label{1'}
\phi(\big[\alpha(\pi), [\zeta, \xi]\big]) = \big[\phi(\pi, \zeta), \alpha(\xi)\big]  + (-1)^{(|\phi| + |\pi|) |\zeta|} [\alpha(\zeta), \phi(\pi, \xi)],
\end{equation}
for all \(\zeta, \xi, \pi \in H\).\\ 
Define 
\[
\phi' := \phi_{H'} : H'\times H' \to H'.
\]
Then, \(\phi'\) is also a super-biderivation of \(H'\).
\begin{thm}\label{thm'}
Let $H$ be a centerless Hom-Lie superalgebra, then
\begin{enumerate}
\item Up to isomorphism, every special super-biderivation of \(H\) is the unique extension of a special super-biderivation of \(H'\).

\item If \(H\) is a Hom-Lie superalgebra with \(H = H'\), then any special super-biderivation of \(H\) is zero.
\end{enumerate}
\end{thm}
\begin{proof} (1) Let \(\phi_1\) and \(\phi_2\) be super-biderivations of \(H\) such that \(\phi_1' = \phi_2'\). Define \(\phi = \phi_1 - \phi_2\).  Then \(\phi(H', H') = 0\). Substituting \(\pi, \xi \in H'\) into Equation (\ref{1'}), we have
\[
[\phi(\pi, \zeta), \alpha(\xi)] = 0, \quad \forall \zeta \in H, \pi, \xi \in H',
\]
which implies \(\phi(H, H') \subseteq Z_{H}(H')\). By Lemma (\ref{lem3.7}), it follows that \([H, \phi(H, H)] \subseteq Z_{H}(H')\).

From Equation (\ref{1''}), we obtain
\[
0 = \big[[\alpha(H), \alpha(H)], \phi(\alpha(H), \alpha(H'))\big] = \big[\phi(\alpha(H), \alpha(H)), [\alpha(H),\alpha(H')]\big].
\]
Thus, for any \(\zeta, \xi, \tau, \pi, \omega \in H\), it suggests that
\begin{align*}
0 &= \Big[\big[[\alpha(\zeta), \alpha(\xi)], \alpha(\tau)\big], \phi(\alpha(\pi), \alpha(\omega))\Big] \\
&= \Big[[\alpha^2(\zeta), \alpha^2(\xi)], [\alpha(\tau), \alpha\phi(\pi, \omega)]\Big] \\
&\quad - (-1)^{|\tau|(|\xi| + |\zeta|)} \Big[\alpha^2(\tau), \big[[\alpha(\zeta), \alpha(\xi)], \phi(\pi, \omega)\big]\Big] \\
&= (-1)^{(|\zeta| + |\xi|) |\tau| + |\tau|(|\zeta| + |\xi| + |\phi| + |\pi| + |\omega|)} \Big[\big[[\alpha(\zeta), \alpha(\xi)], \alpha\phi(\pi, \omega)\big], \alpha^2(\tau)\Big] \\
&= (-1)^{|\tau|(|\phi| + |\pi| + |\omega|)} \Big[\big[[\alpha(\zeta), \alpha(\xi)], \alpha\phi(\pi, \omega)\big], \alpha^2(\tau)\Big]\\
&=(-1)^{|\tau|(|\phi| + |\pi| + |\omega|)} \Big[\big[[\zeta, \xi], \phi(\pi, \omega)\big], \alpha(\tau)\Big].
\end{align*}
 Since, $H$ is centerless, we get that $\phi(H,H)\subset Z_{H}(H')$. Thus, $\phi$ is special super biderivation of $H$.

(2). Let $\phi$ be a special super-biderivation of  $H$. From (1), we have $\phi(H, H')\subseteq Z_H(H')$.
 Since $H=H'$ and $H$  is a centerless Hom-Lie superalgebra, \ we have  $\phi(H, H') = 0$.

Assume that $\phi \neq 0$. Then there exist $\zeta_1$  and $\zeta_2$ in $H$  such that $\phi(\zeta_1, \zeta_2) = \tau_{12} \neq 0$.
 Since $H$  is centerless, pick  $0 \neq \zeta_3 \in H$  such that  $[\alpha(\zeta_3), \tau_{12}] = \tau \neq 0$.
Let $\phi(\zeta_i, \zeta_j) = \tau_{ij}$, where $i, j = 1, 2, 3$. Since $\alpha(H)=H$ and $\phi(H,H')=0$. Then for one thing, we have:
\begin{align*}
0 &= \phi\left(\big[[\zeta_1, \zeta_3], \alpha(\zeta_2)\big]\right) \\
  &= (-1)^{|\phi||\zeta_1|} \big[\alpha(\zeta_1), \phi(\zeta_3, \zeta_2)\big] + (-1)^{|\zeta_3||\zeta_2|} \big[\phi(\zeta_1, \zeta_2), \alpha(\zeta_3)\big] \\
  &= (-1)^{|\phi||\zeta_1|} \big[\alpha(\zeta_1), \tau_{32}\big] + (-1)^{|\zeta_3||\zeta_2|} \big[\tau_{12}, \alpha(\zeta_3)\big] \\
  &= (-1)^{|\phi||\zeta_1|} \big[\alpha(\zeta_1), \tau_{32}\big] - (-1)^{|\zeta_3||\zeta_2| + (|\phi| + |\zeta_1| + |\zeta_2|)|\zeta_3|} \big[\alpha(\zeta_3), \tau_{12}\big] \\
  &= (-1)^{|\phi||\zeta_1|} \big[\alpha(\zeta_1), \tau_{32}\big] - (-1)^{|\zeta_3||\zeta_2| + |\phi||\zeta_3|} \tau.
\end{align*}
Also, we have:
\[
\begin{aligned}
0 &= \phi\left(\big[[\zeta_1, \zeta_2], \alpha(\zeta_3)\big]\right) \\
  &= (-1)^{|\phi||\zeta_1|} \big[\alpha(\zeta_1), \phi(\zeta_2, \zeta_3)\big] + (-1)^{|\zeta_2||\zeta_3|} \big[\phi(\zeta_1, \zeta_3), \alpha(\zeta_2)\big] \\
  &= (-1)^{|\phi||\zeta_1|} \big[\alpha(\zeta_1), \tau_{23}\big] + (-1)^{|\zeta_2||\zeta_3|} \big[\tau_{13}, \alpha(\zeta_2)\big].
\end{aligned}
\]

For another, we get
\[
\begin{aligned}
0 &= \phi\left(\big[[\zeta_2, \zeta_3], \alpha(\zeta_1)\big]\right) \\
  &= (-1)^{|\phi||\zeta_2|} \left[\alpha(\zeta_2), \phi(\zeta_3, \zeta_1)\right] 
  + (-1)^{|\zeta_1||\zeta_3|} \left[\phi(\zeta_2, \zeta_1), \alpha(\zeta_3)\right] \\
  &= (-1)^{|\phi||\zeta_2|} \left[\alpha(\zeta_2), \tau_{31}\right] 
  + (-1)^{|\zeta_1||\zeta_3|} \left[\tau_{21}, \alpha(\zeta_3)\right] \\
  &= (-1)^{|\phi||\zeta_2|} \left[\alpha(\zeta_2), \tau_{31}\right] 
  - (-1)^{|\zeta_1||\zeta_3| + |\zeta_1||\zeta_2|} \left[\tau_{12}, \alpha(\zeta_3)\right] \\
  &= (-1)^{|\phi||\zeta_2|} \left[\alpha(\zeta_2), \tau_{31}\right] 
  + (-1)^{|\zeta_1||\zeta_3| + |\zeta_1||\zeta_2| + (|\phi| + |\zeta_1| + |\zeta_2|)|\zeta_3|} \left[\alpha(\zeta_3), \tau_{12}\right].
  \\
  &= (-1)^{|\phi||\zeta_2|} \left[\alpha(\zeta_2), \tau_{31}\right] 
  + (-1)^{|\zeta_1||\zeta_2| + |\zeta_3||\zeta_2| + |\phi||\zeta_3|} \tau
\end{aligned}
\]
Therefore, it implies that
\[
\begin{aligned}
(-1)^{|\zeta_1||\zeta_3| + |\phi||\zeta_3|} \tau 
&= (-1)^{|\phi||\zeta_1|} [\alpha(\zeta_1), \tau_{32}] \\
&= -(-1)^{|\phi||\zeta_1| + |\zeta_2||\zeta_3|} [\alpha(\zeta_1), \tau_{23}] =[\tau_{13}, \alpha(\zeta_2)] \\
&= -(-1)^{|\zeta_1||\zeta_3|} [\tau_{31}, \alpha(\zeta_2)] \\
&= (-1)^{|\zeta_1||\zeta_3| + (|\phi| + |\zeta_1| + |\zeta_3|)|\zeta_2|} [\alpha(\zeta_2), \tau_{31}] \\
&= -(-1)^{|\phi||\zeta_2| + |\zeta_1||\zeta_2| + |\zeta_2||\zeta_3| + |\zeta_1||\zeta_3| + |\zeta_1||\zeta_2| + |\zeta_2||\zeta_3| + |\phi||\zeta_3| + |\phi||\zeta_2|} \tau \\
&= -(-1)^{|\zeta_1||\zeta_3| + |\phi||\zeta_3|} \tau.
\end{aligned}
\]
Thus, $\tau=-\tau$, which implies $\tau=0$ a contradiction. Hence, $\phi=0$.

\end{proof}

\section{Linear super-commuting maps}
\begin{lem}\label{Lemma 4.1.} 
Let H be a Hom-Lie superalgebra. Suppose $d : H \to H$ is a linear super-commuting map. Then
\[
\Big[[\alpha(\tau), \alpha(\pi)], \big[\alpha(\omega), d([\zeta, \xi]) - [\alpha(\zeta), d(\xi)]\big] \Big] = 0, \quad \forall \, \zeta, \xi, \omega, \tau, \pi \in H.
\]
\end{lem}
\begin{proof}
Since $d$ is a linear super-commuting map, we have
\[
[d(\zeta), \xi] = [\zeta, d(\xi)].
\]
Let
\[
\phi : H \times H \to H
\] such that
\begin{equation}\label{'}
\phi(\zeta, \xi) = [\zeta, d(\xi)]
\end{equation}
then by Lemma (\ref{lem3.1}), $\phi$ is a super-biderivation. From Equation (\ref{'}), we have $|\phi| = |d| = 0$. Now, from Lemma (\ref{lem 3.6})
\begin{equation}\label{''}
\Big[[\alpha(\tau), \alpha(\pi)], \phi(\alpha(\omega), [\zeta, \xi]) - (-1)^{|\phi||\omega|} [\alpha(\omega), \phi(\zeta, \xi)]\Big] = 0,
\end{equation}for all  $\zeta, \xi, \omega, \tau, \pi \in H$.
Since
\[
(-1)^{|\phi||\omega|} [\alpha(\omega), \phi(\zeta, \xi)] = (-1)^{|d||\omega|} [\alpha(\omega), [\zeta, d(\xi)]]
= [\alpha(\omega), [\zeta, d(\xi)]],
\]
and
\[
\phi(\alpha(\omega), [\zeta, \xi])  = [\alpha(\omega), d([\zeta, \xi])], \quad \forall \, \zeta, \xi, \omega, \tau, \pi \in H.
\]
Then from Equation (\ref{''}), we obtain
\[
0 =  \Big[\big[\alpha(\tau), \alpha(\pi)\big], \big[\alpha(\omega), d([\zeta, \xi])\big] - \big[\alpha(\omega), [\zeta, d(\xi)]\big]\Big].
\]
Thus, we conclude
\[
\Big[[\alpha(\tau), \alpha(\pi)], \big[\alpha(\omega), d([\zeta, \xi]) - [\alpha(\zeta), d(\xi)]\big] \Big] = 0, \quad \forall \, \zeta, \xi, \omega, \tau, \pi \in H.
\]
\end{proof}
\begin{thm}
 Suppose $H$ is a Hom Lie superalgebra with $H = H'$ and $Z_H(H') = \{0\}$. If $d: H \to H$ is a linear super-commuting map, then $d \in C_{\bar{0}}(H)$.
 \end{thm}
\begin{proof}
By Lemma (\ref{Lemma 4.1.}), we have
\[
\Big[\big[\alpha(\tau), \alpha(\pi)\big], \big[\alpha(\omega), d([\zeta, \xi]) - [\alpha(\zeta), d(\xi)]\big]\Big] = 0, \quad \forall \zeta, \xi, \pi, \omega, \tau \in H.
\]
Then
\[
\Big[\alpha(\omega), d([\zeta, \xi]) - [\alpha(\zeta), d(\xi)]\Big] \in Z_H(H').
\]
Since \(Z_H(H') = \{0\}\), it follows that
\[
\big[\alpha(\omega), d([\zeta, \xi]) - [\alpha(\zeta), d(\xi)]\big] = 0.
\]
Moreover, we obtain
\[
d([\zeta, \xi]) - [\alpha(\zeta), d(\xi)] \in Z_H(H) = Z_H(H') = \{0\},
\]
that is,
\[
d([\zeta, \xi]) = [\alpha(\zeta), d(\xi)].
\]
Thus \(d \in C_{\bar{0}}(H)\). 
 
\end{proof}

\end{document}